\documentclass[11pt,reqno]{article}

\usepackage[margin=1.15in]{geometry}
\usepackage[T1]{fontenc}
\usepackage{lmodern}
\usepackage{microtype}

\usepackage{amsmath,amssymb,amsthm,mathtools}
\usepackage{mathrsfs}

\usepackage[shortlabels]{enumitem}
\usepackage{tikz-cd}
\usepackage{xcolor}
\usepackage[
  colorlinks=true,
  linkcolor=blue!55!black,
  citecolor=green!40!black,
  urlcolor=blue!65!black
]{hyperref}
\usepackage[nameinlink,capitalise,noabbrev]{cleveref}

\newtheorem{theorem}{Theorem}[section]
\newtheorem{lemma}[theorem]{Lemma}
\newtheorem{proposition}[theorem]{Proposition}
\newtheorem*{proposition*}{Proposition}
\newcommand{\custompropositiontitle}{}
\newtheorem*{custompropositioninner}{\custompropositiontitle}

\newenvironment{customproposition}[1]
  {\renewcommand{\custompropositiontitle}{Proposition #1}%
   \begin{custompropositioninner}}
  {\end{custompropositioninner}}
\newtheorem{corollary}[theorem]{Corollary}
\newtheorem{conjecture}[theorem]{Conjecture}
\newtheorem*{conjecture*}{Conjecture}
\newtheorem*{maintheorem*}{Main Theorem}

\newtheorem{claim}{Claim}[theorem]

\theoremstyle{definition}
\newtheorem{definition}[theorem]{Definition}
\newtheorem{example}[theorem]{Example}

\newtheorem{fact}[theorem]{Fact}

\theoremstyle{remark}
\newtheorem{remark}[theorem]{Remark}
\newtheorem{notation}[theorem]{Notation}

\DeclareMathOperator{\Diag}{Diag}

\DeclareMathOperator{\gr}{gr}

\newcommand{\LCS}{\mathrm{LCS}}
\newcommand{\N}{\mathbb{N}}

\newcommand{\llangle}{\langle\!\langle}
\newcommand{\rrangle}{\rangle\!\rangle}

\title{Existence of a Model Companion for Groups of Exponent $3$}
\author{Yawara Ishida\thanks{A.I. Systems Research Institute Co., Ltd.}, Ryosuke Mizuno\thanks{A.I. Systems Research Institute Co., Ltd.} and Kota Takeuchi\thanks{Institute of Mathematics, University of Tsukuba. Partially supported by JSPS KAKENHI
Grant Number 25K00213. kota@math.tsukuba.ac.jp}}
\date{September, 2026}

\begin{document}

\maketitle

\begin{abstract}
In this article, we prove that the theory $T_3$ of groups of exponent $3$ has a model
companion. Previous work established the existence of model companions
for theories of groups of fixed finite exponent and nilpotency class
at most $2$ (Saracino--Wood), and for theories of groups of prime exponent $p$ and
nilpotency class at most $c<p$ (Maier). These results do not cover $T_3$,
since groups of exponent $3$ may have nilpotency class $3$.
Our theorem verifies the $n=3$ case of a conjecture proposed by the third author that,
for each integer $n>1$, the theory $T_n$ of groups of exponent $n$
has a model companion if and only if every finitely generated group
of exponent $n$ is finite, or equivalently, if and only if the
Burnside problem has a positive solution for exponent $n$.
\end{abstract}



\paragraph{Use of generative AI.}
An initial outline of the proof of the main theorem was generated by ChatGPT (OpenAI, GPT-5.6 Sol).
The authors subsequently developed and reconstructed the argument
and checked all proofs in detail.
The authors take full responsibility for the content of this paper.

\section{Introduction}
\label{sec:introduction}
The question of whether a given theory has a model companion has been studied for many classes of groups.
In the abelian case, many standard theories have model companions.
This is known for the theory of all abelian groups, the theory of torsion-free abelian groups, and the theory of abelian groups of any fixed finite exponent.
More generally, Eklof proved that every inductive theory of abelian
groups with the joint embedding property has a model
companion~\cite{Eklof1972}.

The nonabelian case is less uniform. 
The theory of all groups has no model companion \cite[Example~3.5.16]{ChangKeisler1990}.
The same is true for the theory of torsion-free groups; see, for example, \cite{Takeuchi2022}.
Saracino obtained non-existence results for the theories of solvable groups of fixed derived length at least $2$ \cite{Saracino1974}. 
He also proved that, for every $c\geq2$, neither the theory of nilpotent groups of class at most $c$ nor its torsion-free version has a model companion~\cite{Saracino1976}. 
Thus, standard conditions such as solvability or a fixed bound on nilpotency class are not enough to ensure the existence of a model companion.

Finite exponent changes the answer for some nilpotent classes.
Saracino and Wood proved that the theory of groups of fixed finite exponent and nilpotency class at most $2$ has a model companion~\cite{SaracinoWood1979}.
For prime exponent $p$ and nilpotency class at most $c<p$, Maier proved that the theory has an
$\aleph_0$-categorical model companion~\cite{Maier1989}; see also
\cite[Corollary~4.42]{dElbeeMuellerRamseySiniora2025}.

There are also positive results of a different type.
For example, recent work gives model-complete pure group theories for some semisimple algebraic groups and for Heisenberg groups over model-complete fields~\cite{HoffmannKowalskiTranYe2023,FracekKowalski2025}.
These results show that noncommutativity itself is not an obstruction.
However, they concern complete theories of particular groups.
They do not give a criterion for broad classes of groups defined by simple group-theoretic axioms.
In contrast with the abelian case, no general result is known that describes the boundary between existence and non-existence for such nonabelian classes.

One way to study this boundary is to consider universal theories, and in particular varieties of groups.
The results above suggest that finite exponent is an important condition to examine.
For an integer $n>1$, let $\mathcal V_n$ be the variety of groups satisfying $x^n=1$, and let $T_n$ be its theory in the usual language of groups.
The axioms defining $T_n$ impose a fixed exponent identity but no separate bound on nilpotency class or derived length.
It is therefore a natural test case for the model-companion problem for nonabelian group varieties.

We prove the following.

\begin{maintheorem*}
The theory $T_3$ has a model companion.
\end{maintheorem*}

The case $n=2$ is abelian.
Thus, $n=3$ is the first nonabelian case in this family.
Groups of exponent $3$ are nilpotent of class at most $3$, and this bound is attained~\cite{LeviVanDerWaerden1933}.
Hence our result concerns the whole variety $\mathcal V_3$, including its groups of nilpotency class $3$.
It is not covered by the class-$2$ result of Saracino and Wood or by the results for prime exponent $p$ and class $c<p$.

The family $(T_n)_{n>1}$ is also related to the Burnside problem.
The possible connection between the model-companion problem for groups of fixed exponent and the Burnside problem was already noted by the third author in~\cite{Takeuchi2022}.
The following conjecture, proposed by the third author, gives a precise formulation of this connection.

\begin{conjecture}\label{conj:burnside-model-companion}
For every integer $n>1$, the theory $T_n$ has a model companion if and only if the Burnside problem has a positive solution for exponent $n$.
\end{conjecture}

The condition on the right says that $\mathcal V_n$ is locally finite, or equivalently, every finitely generated group satisfying $x^n=1$ is finite.
The conjecture holds for $n=2$ by the theory of vector spaces over $\mathbb F_2$.
The Burnside problem also has a positive solution for exponent $3$, and the main theorem proves the other side of the conjecture in this case.
We have also obtained a proof that $T_p$ has no model companion for
every sufficiently large prime $p$.
Since the Burnside problem has a negative solution for every
sufficiently large prime $p$ by the Novikov--Adian theorem, and in fact for every odd exponent at least
$665$
\cite{NovikovAdian1968,Adian1979}, this proves the conjecture for all
sufficiently large primes.
The proof will appear elsewhere.

We briefly explain the main idea of the proof.
The variety $\mathcal V_3$ is locally finite, but it does not have the amalgamation property.
We therefore study finite witnesses to failure of amalgamation.
Let $M$ be an existentially closed model of $T_3$, let $A\leq M$, and let $A\leq B\in\mathcal V_3$ be finite.
If $M$ and $B$ have no amalgam over $A$, we show that the same failure already occurs between $B$ and a finite subgroup $D$ of $M$ containing $A$, where the number of generators of $D$ is bounded in terms of the number of generators of $B$.
In other words, the failure of amalgamation in the coproduct $M*_{\mathcal V_3}B$ can be reproduced in a coproduct involving only boundedly many generators.
By Fact~\ref{fact:locally finiteness and model companion}, this uniform bound gives the model companion.

For this purpose, we prove a finite strict-envelope theorem.
Every finite subgroup $C$ of an existentially closed model $M$ is contained in a finite subgroup $D\leq M$ that is strict with respect to the lower central series, with a quadratic bound on the number of generators of $D$.
This strictness implies that the natural map 
\[ D*_{\mathcal V_3}B\longrightarrow M*_{\mathcal V_3}B \]
is injective.
Hence a witness to the failure of amalgamation that is found in $M*_{\mathcal V_3}B$ and involves
only elements of $D$ and $B$ is already a witness in $D*_{\mathcal V_3}B$.
We also prove a structure theorem for existentially closed models of $T_3$: their lower and upper central series coincide in reverse order, and every element of the second and third terms of the lower central series is a commutator or a triple commutator, respectively.
The main tool for these results is the associated graded Lie algebra.
We compute the graded Lie algebras of coproducts in $\mathcal V_3$ and use these computations to construct the required extensions.

Section~\ref{sec:preliminaries} contains the preliminaries.
Section~\ref{sec:main-results} proves the main theorem from the bounded-witness and strict-envelope results.
Section~\ref{sec:structure} proves the structural results used in the argument.
Section~\ref{sec:examples} gives examples showing that lower-central strictness and existential closedness are needed.

We have also fully formalized all the results in this paper in Lean 4.
The formalization depends only on the three standard axioms
(\texttt{propext}, \texttt{Classical.choice}, and \texttt{Quot.sound}),
with no \texttt{sorryAx} or project-specific axioms.
The formalization will be released on GitHub soon.

\section{Preliminaries}
\label{sec:preliminaries}

Throughout, groups are regarded as structures in the language
$\mathcal{L}_{\mathrm{grp}}=\{\cdot,{}^{-1},1\}$.  Let
$T_{\mathrm{grp}}$ denote the theory of groups.  For an integer $n>1$, set
\[
T_n \coloneqq T_{\mathrm{grp}} \cup \{\forall x\,(x^n=1)\}.
\]
Thus, by a group of exponent $n$, we mean a group whose exponent divides $n$.

We fix the following notation for group elements:
\begin{notation}
    \begin{enumerate}
        \item $[a,b]=aba^{-1}b^{-1}$
        \item $[a_1,\ldots,a_k]=[[a_1,\ldots,a_{k-1}],a_k]$
        \item $a^b=b^{-1}ab$
        \item $\mathcal V_3$ is the variety of groups satisfying $x^3=1$; equivalently, $\mathcal V_3=\{M\mid M\models T_3\}$.
        \item $G*H$ is the free product of groups $G$ and $H$.
        \item For a subset $S$ of a group $G$, $\llangle S\rrangle_G$ denotes
        the normal closure of $S$ in $G$.  We omit the subscript when the
        ambient group is clear.
        \item $d(G)$ is the minimum number of generators of $G$. 
        \item $G*_{\mathcal V_3}H$ is the coproduct of $G$ and $H$ in the
        variety $\mathcal V_3$.  Equivalently,
        \[
        G*_{\mathcal V_3}H
        = (G*H)/\llangle (G*H)^3\rrangle_{G*H}.
        \]
        \item For a set $X$, $F_{\mathcal V_3}(X)$ is the free group in the
        variety $\mathcal V_3$ generated by $X$.  Equivalently, if $F(X)$ is
        the free group on $X$, then
        \[
        F_{\mathcal V_3}(X)
        =F(X)/\llangle F(X)^3\rrangle_{F(X)}.
        \]
        If $|X|=r\in\mathbb N$, then $F_{\mathcal V_3}(X)$ is called the
        Burnside group $B(r,3)$.
    \end{enumerate}
\end{notation}

\subsection{Model-theoretic background}

Let $T$ and $T'$ be $L$-theories.
\begin{definition}
    \begin{enumerate}
        \item $T_0$ and $T_1$ are called companions if for $i=0,1$, for every $M_i\models T_i$ there is $M_{1-i}\models T_{1-i}$ such that $M_i\subset M_{1-i}$.
        \item $T$ is said to be model complete if for every formula $\theta(\bar x)$ there is an existential formula $\psi(\bar x)$ such that $T\models \forall \bar x(\theta(\bar x)\leftrightarrow \psi(\bar x))$.
        \item $T^*$ is called a model companion of $T$ if $T^*$ is a companion of $T$ and is model complete.
        \item $M\models T$ is an existentially closed model (e.c. model) of $T$ if for any $\bar a\in M$ and any extension $M\subset N\models T$, $N\models \exists \bar x\varphi(\bar x,\bar a)$ implies $M\models \exists \bar x\varphi(\bar x,\bar a)$ where $\varphi(\bar x,\bar y)$ is quantifier-free.
        \item $T$ is a $\Pi_2$-theory if there is a set $T'$ of $\forall\exists$-sentences such that $T\equiv T'$.
        \item Let $M\models T$ and let $A\subset M$ be a finite subset. Set
        \[
            \Diag_A(\bar x)=\{\theta(\bar x)\mid M\models\theta(A),
            \ \theta(\bar x)\text{ is quantifier-free}\},
        \]
        where $|\bar x|=|A|$. If $L$ is finite and $T$ is locally finite,
        then $\Diag_A(\bar x)$ is finite modulo $T$.
    \end{enumerate}
\end{definition}

\begin{fact}
    Let $T$ be a $\Pi_2$-theory.
    Then for an $L$-theory $T^*$, $T^*$ is a model companion of $T$ if and only if 
    \[
        \{M\mid M\models T^*\}=\{M\models T\mid M\text{ is an e.c. model of $T$}\}.
    \]
\end{fact}

\begin{definition}
    Let $T$ be an $L$-theory.
    We say $T$ is locally finite if, for every $M\models T$ and every finite subset $A\subset M$, the substructure $\langle A\rangle \subset M$ generated by $A$ is finite.
\end{definition}

The next fact follows from a standard compactness argument.
\begin{fact}
    Let $L$ be finite and let $T$ be a locally finite $L$-theory.
    Then for any $n\in\mathbb N$, there is $m\in\mathbb N$ such that if $A\subset M\models T$ is a subset with $|A|\leq n$, then $|\langle A\rangle| \leq m$.
\end{fact}

\begin{fact}\label{fact:locally finiteness and model companion}
    Let $L$ be finite and let $T$ be a locally finite $L$-theory.
    Suppose that $T$ is a $\Pi_2$-theory.
    Then the following are equivalent.
    \begin{enumerate}
        \item $T$ has a model companion $T^*$.
        \item For any finite substructures $A\subset B$ of a model $M_0\models T$, there is $n_{A,B}\in \mathbb N$ satisfying the following: 
        Let $M$ be an e.c. model of $T$ and let $A\subset M$. 
        If $M$ and $B$ do not have an amalgam over $A$ in $T$, then there is $A\subset C\subset M$ such that
        \begin{enumerate}
            \item $C$ is generated by at most $n_{A,B}$ elements as an $L$-substructure,
            \item $C$ and $B$ do not have an amalgam over $A$ in $T$.
        \end{enumerate}   
    \end{enumerate}
\end{fact}

\begin{proof}
    For given finite $A\subset B$, let $\theta_{B,A}(\bar x,A)$ be a quantifier-free $L(A)$-formula expressing $\mathrm{Diag}_{AB}(\bar x,\bar y)$; in other words, $\theta_{B,A}(B',A')$ holds if and only if $A\cong A'$ and $B \cong_A B'$ after identifying $A'$ with $A$. Since $L$ is finite, we can find such a formula $\theta_{B,A}(\bar x,A)$.

    ($1\Rightarrow 2$)
    Since $T^*$ is a model companion, $T^*\models \forall \bar y (\neg\exists \bar x\theta_{B,A}(\bar x,\bar y) \leftrightarrow  \exists \bar z \psi(\bar z, \bar y))$ for some quantifier-free $L$-formula $\psi(\bar z,\bar y)$.
    Let $n_{A,B}=|\bar z|+|A|$.

    Suppose $M$ is an e.c. model of $T$ and $A\subset M$.
    If $M$ and $B$ do not have an amalgam over $A$ in $T$, then $M\models \neg \exists \bar x\theta_{B,A}(\bar x,A)$.
    Since $M\models T^*$, $M\models \exists \bar z\psi(\bar z,A)$.
    Let $\bar c\in M$ be a realization of $\psi(\bar z,A)$ in $M$ and let $C=\langle \bar c,A\rangle_M\subset M$.
    Suppose that $C$ and $B$ have an amalgam $N_0\models T$ over $A$. We can extend $N_0\subset N\models T^*$.
    Then $N\models \psi(\bar c, A)$, hence $N\models \neg\exists \bar x\theta_{B,A}(\bar x,A)$. 
    However, $N$ contains a copy of $B$ over $A$, which implies $\exists \bar x \theta_{B,A}(\bar x,A)$. 

    $(2\Rightarrow 1)$
    For each finite $A\subset B$, let $K(A,B)$ be the set of all (isomorphism types of) $C\supset A$ over $A$ such that
    \begin{enumerate}
        \item $C$ is a substructure of some $N\models T$ and $C$ is generated by at most $n_{A,B}$ elements,
        \item $C$ and $B$ do not have an amalgam in $T$ over $A$.
    \end{enumerate}
    Because $T$ is locally finite, the size $|C|$ is uniformly bounded by some $m_{A,B}\in\N$. 
    Hence $K(A,B)$ is finite since $L$ is finite.
    Let $\theta_A(\bar y)$ be a quantifier-free formula expressing $\mathrm{Diag}_A(\bar y)$, and let
    \[
        T^*=T\cup \left\{\forall \bar y\left[\theta_A(\bar y)\wedge\left(\bigwedge_{C\in K(A,B)}\neg\exists \bar z\theta_{C,A}(\bar z, \bar y)\right) \to \exists \bar x\theta_{B,A}(\bar x,\bar y)\right] \mid A\subset B\text{ are finite}\right\}. 
        \]
    If $M\models T$ is an e.c. model, then clearly $M\models T^*$, because $A\subset M$ and $A\subset B$ have an amalgam in $T$ over $A$ if and only if $M\models\exists \bar x\theta_{B,A}(\bar x,A)$.
    Conversely, suppose $M\models T^*$ and let $\bar a\subset M\subset N\models T$, $\varphi(\bar x,\bar a)$ a quantifier-free formula, and $N\models \exists \bar x\varphi(\bar x, \bar a)$.
    Then there is $\bar b\in N$ such that $N\models \varphi(\bar b,\bar a)$, so let $B=\langle \bar b,\bar a\rangle_N$ and $A=\langle \bar a\rangle_M$.
    This means $N$ is an amalgam of $M$ and $B$ over $A$.
    For every $C\in K(A,B)$, the definition of $K(A,B)$ gives $N\models \neg\exists \bar z\theta_{C,A}(\bar z, A)$.
    Because $M\subset N$ and $\theta_{C,A}$ is quantifier-free, $M\models \neg\exists \bar z\theta_{C,A}(\bar z, A)$. Hence we can find $B'\subset M\models \theta_{B,A}(B',A)$ and $\bar b'\in B'$ such that $M\models\varphi(\bar b',\bar a)$.
\end{proof}

\subsection{Graded Lie algebras and block decomposition}
\begin{definition}
    A graded Lie ring $L$ is an abelian group
    $(L,+)$ with a decomposition $L=\bigoplus_{i\geq 1}L_i$ and a Lie
    bracket $[-,-]:L\times L\to L$ satisfying the following:
    \begin{enumerate}
        \item $[L_i,L_j]\subset L_{i+j}$
        \item $[-,-]$ is biadditive:$[a+b,c]=[a,c]+[b,c]$ and $[a,b+c]=[a,b]+[a,c]$
        \item $[a,a]=0$
        \item Jacobi identity:$[[a,b],c]+[[b,c],a]+[[c,a],b]=0$
    \end{enumerate}
    If $L$ is a vector space over a field $k$ and $[-,-]$ is $k$-bilinear,
    then $L$ is called a graded Lie algebra over $k$.
\end{definition}

\begin{remark}
    \begin{enumerate}
        \item $[a,b]=-[b,a]$ for $a,b\in L$.
        \item $L$ is generated by $L_1$ if and only if
        $[L_i,L_1]=L_{i+1}$ for every $i\geq1$.
    \end{enumerate}
\end{remark}

\begin{example}[Grassmann algebra]\label{example:Grassmann algebra}
    Let $V$ be a vector space over $\mathbb F_3$, and set
    \[
    \Lambda^{[1,3]}V
    \coloneqq \Lambda^1V\oplus\Lambda^2V\oplus\Lambda^3V.
    \]
    For homogeneous elements $u\in\Lambda^iV$ and $v\in\Lambda^jV$, define
    \[
    [u,v]\coloneqq
    \begin{cases}
        u\wedge v,&(i,j)=(1,1)\text{ or }(2,1),\\
        -u\wedge v,&(i,j)=(1,2),\\
        0,&i+j\geq4.
    \end{cases}
    \]
    Extending this bracket bilinearly makes $\Lambda^{[1,3]}V$ a graded Lie
    algebra over $\mathbb F_3$.
\end{example}

\begin{definition}
    Let $V$ be a vector space and fix a decomposition $V=\bigoplus_\lambda V_\lambda$.
    Each $V_\lambda$ is called a block or $\lambda$-block.
    A subspace $W\subset V$ is called block-homogeneous (with respect to the decomposition) if $W=\bigoplus_\lambda (W\cap V_\lambda)$.
\end{definition}

\begin{proposition}
    Let $V=\bigoplus_\lambda V_\lambda$ be a decomposition of a vector space and let $W\subset V$ be a block-homogeneous subspace of $V$.
    Then $V/W\cong\bigoplus_\lambda V_\lambda/(W\cap V_\lambda)$ with a standard isomorphism $\bar v \mapsto \bigoplus_\lambda \bar v_\lambda$ where $v=\bigoplus_\lambda v_\lambda \in V$.
\end{proposition}
\begin{proof}
    Straightforward.
\end{proof}

\subsection{Group-theoretic background}
Let $G$ be a group.
For subsets $A,B\subset G$, let $[A,B]$ denote the subgroup generated by
$\{[a,b]:a\in A,\ b\in B\}$.
\begin{definition}[Lower and upper central series]
    \begin{enumerate}
        \item The lower central series is defined by $\gamma_1(G)=G$ and
        \[
            \gamma_{n+1}(G)=[\gamma_n(G),G] \qquad (n\geq 1).
        \]
        \item The upper central series is given by $Z_0(G)=1$ and
        \[
            Z_n(G)=\{a\in G:[a,g_1,\ldots,g_n]=1
            \text{ for all }g_1,\ldots,g_n\in G\}
            \qquad (n\geq 1).
        \]
    \end{enumerate}
\end{definition}
\begin{remark}
    This description of $Z_n$ is equivalent to the usual recursive definition of the upper central series. Also, this form is useful for our argument.
\end{remark}

\begin{fact}
    \begin{enumerate}
        \item $[\gamma_i(G),\gamma_j(G)]\leq \gamma_{i+j}(G)$, hence $\gamma_i(G)/\gamma_{i+1}(G)$ is abelian.
        \item If $G$ is nilpotent of class at most $n$, then $\gamma_i(G)\leq Z_{n+1-i}(G)$ for every $1\leq i\leq n$.
    \end{enumerate}
\end{fact}

In what follows, we assume groups are of exponent $3$ unless otherwise noted.

\begin{fact}\label{fact:elementary equations}
Let $G$ be an exponent $3$ group and $a,b,c,d\in G$.
The following are well known and elementary.
\begin{enumerate}
    \item $[a,b,c,d]=1$; equivalently, $G$ is nilpotent of class at most $3$.
    \item $\gamma_3(G)\leq Z(G)$ and $\gamma_2(G)\leq Z_2(G)$.
    \item $\gamma_2(G)$ is abelian, but $Z_2(G)$ may not be abelian in general. 
    \item $[a,b,c]=[b,c,a]=[c,a,b]$. In particular, $[a,b,b]=1$; equivalently, $G$ is $2$-Engel.
    \item $[a,b^{-1}]=[a,b]^{-1}=[a^{-1},b]$
    \item $[a^b,a^c]=1$ where $a^b=b^{-1}ab$. Hence the normal closure $\llangle a\rrangle_G$ is abelian.
    \item $[a,b,c]=[b,a,c]^{-1}$
    \item $[a,bc]=[a,b][a,c][a,b,c]$
    \item $[ab,c]=[a,c][b,c][a,b,c]^{-1}$
    \item $[A,BC]\subset [A,B][A,C][A,B,C]$
    \item $[AB,C] \subset [A,C][B,C][A,B,C]$
\end{enumerate}
\end{fact}

The next fact is also obtained by standard calculations of groups.
\begin{fact}
Let $\mathcal I$ and $\mathcal J$ be disjoint sets. Then $F_{\mathcal V_3}(\mathcal I\cup \mathcal J)=F_{\mathcal V_3}(\mathcal I)*_{\mathcal V_3}F_{\mathcal V_3}(\mathcal J)$.
\end{fact}

\begin{lemma}
    Let $G,H\in\mathcal V_3$.
    Then natural mappings $j_G:G\to G*_{\mathcal V_3}H$ and $j_H:H\to G*_{\mathcal V_3}H$ are injective.
    Hence we can regard $G,H\leq G*_{\mathcal V_3}H$.
\end{lemma}
\begin{proof}
    Define $f:G*_{\mathcal V_3} H\to G$ by $f(j_H(h))=1$ for all $h\in H$ and $f(j_G(g))=g$ for all $g\in G$. 
    Then $f$ is a retraction of $j_G$ and hence $j_G$ is injective.
    By symmetry, $j_H$ is also injective.
\end{proof}

 We recall the associated graded Lie algebra $\gr(G)$ for a group $G$ of exponent $3$.
\begin{definition}
    Let $G$ be a group of exponent $3$.
    \begin{enumerate}
        \item $\gr_n(G) \coloneqq \gamma_n(G)/\gamma_{n+1}(G)$.
        \item For $a\in\gamma_n(G)$ and $b\in\gamma_m(G)$, define
        \[
        [a\gamma_{n+1}(G),b\gamma_{m+1}(G)]
        \coloneqq [a,b]\gamma_{n+m+1}(G).
        \]
    \end{enumerate}
    Then $\gr(G)=\bigoplus_i\gr_i(G)$ is a graded Lie algebra.
    Indeed, the inclusions
    $[\gamma_i(G),\gamma_j(G)]\leq\gamma_{i+j}(G)$ show that the bracket is
    independent of the chosen representatives.
    Each quotient $\gr_n(G)$ has exponent $3$, and hence is naturally
    a vector space over $\mathbb F_3$.
    The remaining axioms follow from Fact \ref{fact:elementary equations}.
\end{definition}

The following are also immediate by Fact \ref{fact:elementary equations}.
\begin{lemma}
Let $G$ be a group of exponent $3$.
    \begin{enumerate}
        \item $\gr_4(G)=0$.
        \item $[a,b,c]=[b,c,a]$ and $[a,b,b]=0$ for every
        $a,b,c\in\gr(G)$.
        \item $\gr(G)$ is generated by $\gr_1(G)$ as a Lie algebra over
        $\mathbb F_3$.
    \end{enumerate}
\end{lemma}

\begin{notation}
    \begin{enumerate}
        \item For $a\in \gamma_i(G)$, we define
        $\gr_i(a)=a\gamma_{i+1}(G)\in\gr_i(G)$.
        \item For $K\leq G$, 
        \[
        \gr_n^G(K)
        \coloneqq
        \{\gr_n(a)\mid a\in \gamma_n(G)\cap K\}
        =
        \frac{(K\cap\gamma_n(G))\gamma_{n+1}(G)}{\gamma_{n+1}(G)}
        \leq \gr_n(G).
        \]
        Hence $\gr_n^G(K)\cong (K\cap \gamma_n(G))/(K\cap \gamma_{n+1}(G))$. 
    \end{enumerate}
\end{notation}

\begin{example}
    \begin{enumerate}
        \item If $G$ is abelian, then $\gr(G)=\gr_1(G)=G$.
        \item If $G=H\times K$, then $\gamma_i(G)=\gamma_i(H)\times \gamma_i(K)$ and $\gr_i(G)\cong \gr_i(H)\oplus \gr_i(K)$.
        \item Let $F_2:= F_{\mathcal V_3}(\{x,y\})$ be a free group in $\mathcal V_3$ generated by $\{x,y\}$. Then $\gamma_2(F_2)=\{[x,y]^k\mid 0\leq k\leq 2\}$ and $\gamma_3(F_2)=1$, hence $\gr_1(F_2)=\mathbb F_3\bar x\oplus \mathbb F_3\bar y$ and $\gr_2(F_2)=\mathbb F_3 [x, y]$.
    \end{enumerate}
\end{example}

If $f:G\to H$ is a group homomorphism, then $\gamma_i(f(G))\leq \gamma_i(H)$ holds.
So we can consider an induced homomorphism between associated graded Lie algebras.
\begin{definition}
    For a group homomorphism $f:G\to H$, put
    $\gr(f)=\bigoplus_{i\geq1}\gr_i(f):\gr(G)\to\gr(H)$, where
    \[
    \gr_i(f)\bigl(a\gamma_{i+1}(G)\bigr)
    =f(a)\gamma_{i+1}(H).
    \]
\end{definition}

We can describe when $\gr(f)$ is injective by using strict extensions given in the next definition.

\begin{definition}
    Let $G\leq H$ be groups of exponent $3$.
    We say that $G$ is \emph{strict with respect to the lower central series}
    in $H$, and write $G\leq_{\LCS}H$, if
    \begin{enumerate}
        \item $\gamma_2(G)=\gamma_2(H)\cap G$ and
        \item $\gamma_3(G)=\gamma_3(H)\cap G$.
    \end{enumerate}
\end{definition}

\begin{proposition}\label{proposition:gr(f) and LCS}
    Let $f:G\to H$ be a homomorphism of exponent $3$ groups. 
    \begin{enumerate}
        \item If $\gr(f):\gr(G)\to\gr(H)$ is injective then so is $f$.
        \item If $f$ is injective, then the following are equivalent.
        \begin{enumerate}
            \item $\gr(f):\gr(G)\to\gr(H)$ is injective.
            \item $f(G)\leq_{\LCS} H$.
        \end{enumerate}
    \end{enumerate}
\end{proposition}
\begin{proof}
    For (1), suppose that $f$ is not injective and choose $1\neq a\in\ker(f)$.
    Since $\gamma_4(G)=1$, there is a largest $i\in\{1,2,3\}$ such that
    $a\in\gamma_i(G)$.  Then
    $a\gamma_{i+1}(G)$ is a nonzero element of $\gr_i(G)$ that belongs to
    $\ker(\gr_i(f))$.  Hence $\gr(f)$ is not injective.

    For (2), assume that $f$ is injective and identify $G$ with $f(G)\leq H$.
    For $i=1,2,3$, one has
    \[
    \ker(\gr_i(f))
    =
    \frac{\gamma_i(G)\cap\gamma_{i+1}(H)}{\gamma_{i+1}(G)}.
    \]
    Thus $\gr_1(f)$ is injective if and only if
    $G\cap\gamma_2(H)=\gamma_2(G)$.  Under this equality,
    $\gr_2(f)$ is injective if and only if
    $G\cap\gamma_3(H)=\gamma_3(G)$.  Finally, $\gr_3(f)$ is injective because
    $\gamma_4(H)=1$.  These are exactly the two conditions defining
    $G\leq_{\LCS}H$.
\end{proof}

\begin{definition}
    Let $G\in \mathcal V_3$. 
    We say that the lower central series of $G$ coincides with the upper
    central series in reverse order if $\gamma_i(G)=Z_{4-i}(G)$ for every
    $1\leq i\leq 3$.
\end{definition}

\begin{lemma}
    Let $G\leq H\in \mathcal V_3$.
    \begin{enumerate}
        \item $\gamma_i(G)\subset \gamma_i(H)\cap G \subset Z_{4-i}(H)\cap G \subset Z_{4-i}(G)$ for $1\leq i\leq 3$.
        \item If the lower central series of $G$ coincides with the upper
        central series in reverse order, then $G\leq_{\LCS}H$. 
    \end{enumerate}
\end{lemma}
\begin{proof}
    (1) is immediate from the definition of $\gamma_i$ and $Z_{4-i}$, because $\gamma_4(H)=1$.

    (2) is an easy application of (1).
\end{proof}

To determine an associated graded Lie algebra $\gr(F)$ of $F=F_{\mathcal V_3}(\mathcal I)$, we recall the following result:
\begin{fact}[Levi and van der Waerden \cite{LeviVanDerWaerden1933}]\label{fact:Levi and van der Waerden}
    Let $\mathcal I=\{x_0,\cdots, x_{r-1}\}$ and $F=F_{\mathcal V_3}(\mathcal I)=B(r,3)$ the Burnside group of $r$-generators and exponent $3$. Then $|F|=3^{t(r)}$ where $t(r)=r + \binom r2 + \binom r3$, and every $a\in F$ is uniquely expressed as follows.
    \[
    a=x_0^{l_0}\cdots x_{r-1}^{l_{r-1}}
    \prod_{i<j<r}[x_i,x_j]^{m_{i,j}}
    \prod_{i<j<k<r}[x_i,x_j,x_k]^{n_{i,j,k}}
    \quad (l_i,m_{i,j},n_{i,j,k}\in\mathbb F_3)
    \]
\end{fact}

\begin{remark}\label{remark:infinite dim}
    When $\mathcal I=\{x_i\mid i< \kappa\}$ for some infinite cardinal $\kappa$, then each $a\in F_{\mathcal V_3}(\mathcal I)$ has a unique expression
    \[
    a=\prod_{i<\kappa} x_i^{l_i}
    \prod_{i<j<\kappa}[x_i,x_j]^{m_{i,j}}
    \prod_{i<j<k<\kappa}[x_i,x_j,x_k]^{n_{i,j,k}}
    \]
    where $l_i,m_{i,j},n_{i,j,k}\in\mathbb F_3$ and the coefficient families
    $(l_i)$, $(m_{i,j})$, and $(n_{i,j,k})$ have finite support.
    This follows from the identity $F_{\mathcal V_3}(\mathcal I\sqcup \mathcal J)=F_{\mathcal V_3}(\mathcal I)*_{\mathcal V_3}F_{\mathcal V_3}(\mathcal J)$.
\end{remark}

We use this normal form to identify the associated graded Lie algebra of
$B(r,3)$.

\begin{proposition}\label{proposition:gr(F) is Grassmann algebra}
    Let $F$ be the Burnside group $B(r,3)=F_{\mathcal V_3}(\{x_0,\cdots, x_{r-1}\})$.
    Let $V$ be a $\mathbb F_3$-vector space with basis
    $\{x_0,\ldots,x_{r-1}\}$.  Define linear maps
    $\sigma_n:\gr_n(F)\to\Lambda^nV$ by
    \begin{enumerate}
        \item $\overline{x_i}\mapsto x_i$ for $n=1$,
        \item $\overline{[x_i,x_j]}\mapsto x_i\wedge x_j$ for $n=2$ and
        $i<j$,
        \item $\overline{[x_i,x_j,x_k]}\mapsto x_i\wedge x_j\wedge x_k$
        for $n=3$ and $i<j<k$,
    \end{enumerate}
    where $\bar a=a\gamma_{i+1}(F)\in \gr_i(F)$ for $a\in\gamma_i(F)$.
    Then
    \[
    \sigma=\sigma_1\oplus\sigma_2\oplus\sigma_3:
    \gr(F)\longrightarrow\Lambda^{[1,3]}V
    \]
    is an isomorphism of graded Lie algebras, where the bracket on
    $\Lambda^{[1,3]}V$ is the one defined in Example \ref{example:Grassmann algebra}.
\end{proposition}

\begin{proof}
    By Fact \ref{fact:Levi and van der Waerden}, the following are bases over
    $\mathbb F_3$:
    $$
    \{\overline{x_i}:0\leq i<r\}\quad\text{for }\gr_1(F),
    $$
    \[
    \{\overline{[x_i,x_j]}:i<j\}\quad\text{for }\gr_2(F),
    \]
    and
    \[
    \{\overline{[x_i,x_j,x_k]}:i<j<k\}\quad\text{for }\gr_3(F).
    \]
    Hence each $\sigma_n$ is a linear isomorphism.  The definition of the
    bracket gives
    \[
    \sigma_2([\overline{x_i},\overline{x_j}])=x_i\wedge x_j
    \]
    and
    \[
    \sigma_3([\overline{[x_i,x_j]},\overline{x_k}])
    =x_i\wedge x_j\wedge x_k.
    \]
    The brackets in the opposite order are determined by alternation, and all
    brackets of total degree at least $4$ vanish.  Thus $\sigma$ preserves the
    Lie bracket.
\end{proof}

\begin{remark}
The above result remains true when $F$ is generated by an infinite set $\{x_i\mid i< \kappa\}$ for some cardinal number $\kappa$. See Remark \ref{remark:infinite dim}. 
\end{remark}

Using this isomorphism, we identify $\gr_i(B(r,3))$ with $\Lambda^iV$ for
$i=1,2,3$.

Now let $G\in \mathcal V_3$ be a finite group generated by $\{g_i\mid 1\leq i\leq r\}\subset G$.
Then $G$ is a quotient of $F_{\mathcal V_3}(\{x_1,\ldots,x_r\})$, via the
homomorphism induced by $x_i\mapsto g_i$.
Hence, in order to understand $\gr(G)$, we need to know how to calculate the associated graded Lie algebra $\gr(F/N)$.

Now we calculate the associated graded Lie algebras of quotient groups.

Recall that
\[
\gr_i^G(N)
=\frac{(N\cap\gamma_i(G))\gamma_{i+1}(G)}{\gamma_{i+1}(G)}
\leq\gr_i(G)
\]
for $N\leq G$.
\begin{lemma}\label{lemma:gr of quotient}
    Let $G$ be a group of exponent $3$ and let $N\trianglelefteq G$.
    Then there is a canonical isomorphism $\gr_i(G/N)\cong \gr_i(G)/\gr_i^G(N)$.
\end{lemma}

\begin{proof}
    Since $\gamma_i(G/N)=\gamma_i(G)N/N$, 
    \[
    \gr_i(G/N)
    \cong
    \frac{\gamma_i(G)N}{\gamma_{i+1}(G)N}
    \cong
    \frac{\gamma_i(G)}{\gamma_{i}(G)\cap(\gamma_{i+1}(G)N)}
    =
    \frac{\gamma_i(G)}{\gamma_{i+1}(G)(\gamma_i(G)\cap N)}
    \cong
    \frac{\gr_i(G)}{\gr_i^G(N)}.
    \]
\end{proof}

\section{Main results}
\label{sec:main-results}
In this section we prove that $T_3$ has a model companion, assuming Proposition A, which is proved in the next section.

\begin{proposition}\label{proposition:bounded number of conjugates}
    For any $a\in G\in \mathcal V_3$,
    \[
    \llangle a\rrangle_G = \{a^{g_0}\cdots a^{g_{n-1}}\mid g_i\in G, n\leq 3\}.    
    \]
    In other words, every element in $\llangle a\rrangle_G$ is a product of at most $3$ conjugates of $a$. 
\end{proposition}
\begin{proof}
Let $E_a=\{a^k[a,g]\mid g\in G, k\in\mathbb F_3\}\subset \llangle a\rrangle$.
Because $[a,g][a,g']=[a,gg'[g,g']]$, $[a,g]^{-1}=[a,g^{-1}]$ and $a[a,g]=[a,g]a$, $E_a$ is a subgroup.
Since $a^g=g^{-1}ag=a[a^{-1},g^{-1}]=a[a,g]$, $E_a=\llangle a\rrangle_G$.
Since $a[a,g]=a^g$, $a^2[a,g]=aa^g$, and $[a,g]=a^{-1}a^g=a^2a^g$, every element in $E_a$ is a product of at most $3$ conjugates of $a$.
\end{proof}

\begin{lemma}\label{lemma:witness in bdd support}
    Let $B,G\leq H\in \mathcal V_3$, $d(B)\leq m$, $\langle B,G\rangle = H$ and $\Delta\subset H$ with $|\Delta|\leq n$.
    Then for any $h\in \llangle \Delta \rrangle_H$,
    there is $C\leq G$ such that
    \begin{enumerate}
        \item $d(C)\leq 3(m+1)n$,
        \item $h\in \llangle \Delta \rrangle_{H_0}$ where $H_0=\langle C,B,\Delta\rangle\leq H$.
    \end{enumerate}
    
\end{lemma}
\begin{proof}
Let $\Delta=\{\delta_i\mid i<r\}$ with $r\leq n$. Then $\llangle \Delta \rrangle_H=\llangle \delta_0\rrangle_H\cdots\llangle\delta_{r-1}\rrangle_H$, so it suffices to prove the claim for $\Delta=\{\delta\}$.

By Proposition \ref{proposition:bounded number of conjugates}, $h=\prod_{i<k}\delta^{u_i}$ for some $u_i\in H$ and $k\leq 3$.
Since $H=\langle G,B\rangle$, every $u\in H$ is expressed as
\[
    u = gb\prod_{j<m}[g_j,b_j]^{k_j} \mod \gamma_3(H)
    \]
for some $g,g_j\in G$, $b\in B$ and $k_j\in \mathbb F_3$ where $B=\langle b_j\mid j<m\rangle$.
Here, we used that $[gg',b]=[g,b][g',b][g,g',b]^{-1}$, $[b,g]=[g,b]^{-1}$ and $[H,H]$ is abelian.
Because $\gamma_3(H)\leq Z(H)$, $\delta^u=\delta^v$ where $u=v \mod\gamma_3(H)$.

Hence $C$ is obtained by collecting all $g,g_j\in G$ in the above expression for each $u_i$ $(i<k)$ as generators of $C$.

This argument shows that $C$ is generated by at most $k(m+1)|\Delta|\leq k(m+1)n\leq 3(m+1)n$ elements.
\end{proof}

A proof of the next proposition is given in the next section as Proposition \ref{proposition:bdd LCS}.
\begin{customproposition}{A}
    There is a function $f_0:\mathbb N\to\mathbb N$ such that for every $n\in\mathbb N$, every e.c. model $M$ of $T_3$, and every $A\leq M$ with $d(A)\leq n$, there is $A\leq D\leq_{\LCS}M$ such that $D$ is generated by at most $f_0(n)$ elements.
\end{customproposition}

\begin{theorem}[Main theorem]
\label{thm:main}
There is a function $f:\mathbb N\to \mathbb N$ satisfying the following:
Let $T_3$ be the theory of groups of exponent $3$ and $M$ an existentially closed model of $T_3$.
Let $A\leq M$ be a finite subgroup and let $A\leq B\in\mathcal V_3$ be such that $d(B)\leq m$.
If $M$ and $B$ do not have an amalgam in $\mathcal V_3$ over $A$, then there is $A\leq D\leq M$ such that 
\begin{enumerate}
    \item $D$ and $B$ do not have an amalgam in $\mathcal V_3$ over $A$.
    \item $d(D)\leq f(m)$.
\end{enumerate}
\end{theorem}

\begin{proof}
Suppose that $A\leq M$ and $A\leq B$ do not have an amalgam over $A$. 
Let $H=M*_{\mathcal V_3}B$ and let $j_M:M\to H$ and $j_B:B\to H$ be canonical embeddings.
Let $d(A)=n$ and let $A$ be generated by $\{a_0,\cdots, a_{n-1}\}$. 
Put $\Delta=\{\delta_{a_i}\mid i< n\}\subset H$ where $\delta_a=j_M(a)j_B(a)^{-1}$.
By the assumption, at least one of the canonical homomorphisms $\bar j_M:M\to H/\llangle \Delta\rrangle_H$ and $\bar j_B:B\to H/\llangle \Delta \rrangle_H$ is not injective; equivalently, at least one of the following holds:
\begin{itemize}
    \item There is $1\neq g\in M$ such that $j_M(g)\in \llangle \Delta \rrangle_H$.
    \item There is $1\neq g\in B$ such that $j_B(g)\in \llangle \Delta \rrangle_H$.
\end{itemize}
Since the other case is similar, we assume $g\in M$ and $j_M(g)\in \llangle \Delta \rrangle_H$.

Recall that $d(A)=n$ and $d(B)\leq m$.
By Lemma \ref{lemma:witness in bdd support},
there is $C_0\leq M$ generated by $3(m+1)n$ elements such that, by letting $H_0=\langle j_M(C_0),j_B(B),\Delta\rangle \leq H$, $j_M(g)\in \llangle \Delta\rrangle_{H_0}$.
Let $C=\langle C_0,A,g\rangle_M$, which is generated by $3(m+1)n+n+1$ elements.
Since $n= d(A)\leq \log_3(|A|) \leq \log_3(|B|) \leq t(m)$ where $t(m)=m+\binom m2+\binom m3$,
we have $3(m+1)n+n+1\leq (3m+4)t(m)+1$.
By Proposition A, we can find $C\leq D\leq_\LCS M$ such that $D$ is generated by at most $f(m):=f_0((3m+4)t(m)+1)$ elements.

\begin{claim}
$D$ and $B$ do not have an amalgam in $\mathcal V_3$ over $A$.
\end{claim}
Since $D\leq_\LCS M$, a natural homomorphism $D*_{\mathcal V_3}B \to M*_{\mathcal V_3}B$ is an embedding, so we identify $D*_{\mathcal V_3}B$ as a subgroup of $M*_{\mathcal V_3}B$.
Since $j_M|D$ agrees with the natural embedding $j_D:D\to D*_{\mathcal V_3}B\leq M*_{\mathcal V_3}B$ and $H_0\leq \langle j_D(D),j_B(B)\rangle\leq D*_{\mathcal V_3}B$, we know $1\neq j_D(g)=j_M(g)\in \llangle \Delta\rrangle_{D*_{\mathcal V_3}B}\leq D*_{\mathcal V_3}B$.
Hence $D\to D*_{\mathcal V_3}B/\llangle \Delta\rrangle_{D*_{\mathcal V_3}B}$ is not an embedding. This shows that $A\leq D$ and $A\leq B$ do not have an amalgam over $A$.
\end{proof}

By Fact \ref{fact:locally finiteness and model companion}, we have the following.
\begin{corollary}
    Let $T_3$ be the theory of groups of exponent $3$. Then $T_3$ has a model companion.
\end{corollary}

\section{Structural analysis}
\label{sec:structure}

We first construct a relatively free presentation of $G$.

Let $G$ be a group of exponent $3$.
Since $V_G=G/\gamma_2(G)$ is abelian, it is a vector space over $\mathbb F_3$.
Let $\mathcal I=\{e_i\mid i\in I\}$ be a basis of $V_G$.  For each $i\in I$,
choose a representative $a_i\in G$ such that
$e_i=a_i\gamma_2(G)$.

\begin{proposition}\label{proposition:lift}
    The assignment $e_i\mapsto a_i$ induces a surjective homomorphism
    $\varphi_G:F_{\mathcal V_3}(\mathcal I)\to G$ such that
    $\ker(\varphi_G)\leq\gamma_2(F_{\mathcal V_3}(\mathcal I))$.
\end{proposition}

\begin{proof}
    Let $F=F_{\mathcal V_3}(\mathcal I)$ and
    $H=\langle a_i\mid i\in I\rangle\leq G$.  Since the cosets
    $a_i\gamma_2(G)$ form a basis of $G/\gamma_2(G)$, we have
    $G=H\gamma_2(G)$.

    We claim that $H=G$.  Since $\gamma_2(G)$ is abelian and
    $\gamma_3(G)\leq Z(G)$, the commutator identities give
    \[
    \gamma_2(G)
    =[H\gamma_2(G),H\gamma_2(G)]
    \leq H[H,\gamma_2(G)]
    \leq H\gamma_3(G).
    \]
    Consequently,
    \[
    [H,\gamma_2(G)]
    \leq[H,H\gamma_3(G)]
    =[H,H]
    \leq H,
    \]
    where we used the centrality of $\gamma_3(G)$.  It follows that
    $\gamma_2(G)\leq H[H,\gamma_2(G)]\leq H$, and hence
    $G=H\gamma_2(G)=H$.  Thus $\varphi_G$ is surjective.

    The induced map
    \[
    \overline{\varphi}_G:F/\gamma_2(F)\longrightarrow G/\gamma_2(G)
    \]
    sends the basis $\mathcal I$ bijectively to the basis
    $\{a_i\gamma_2(G):i\in I\}$, so it is an isomorphism.  If
    $a\in\ker(\varphi_G)$, then
    $a\gamma_2(F)\in\ker(\overline{\varphi}_G)=0$.  Therefore
    $a\in\gamma_2(F)$, as required.
\end{proof}

\begin{remark}\label{remark:G=H}
    The above argument shows that, in general, if $H\leq G$ and $G=H\gamma_2(G)$, then $G=H$.
\end{remark}

\begin{lemma}\label{lemma:gr of normal closure}
    Let $G$ be a group of exponent $3$, let $K\leq\gamma_2(G)$, and let
    $L=\llangle K\rrangle_G$ be the normal closure of $K$ in $G$.
    Then
    \begin{enumerate}
        \item $L=K[K,G]$ and $L\leq \gamma_2(G)$,
        \item $\gr_2^G(L)=\gr_2^G(K)=\{\gr_2(a)\in \gr_2(G)\mid a\in K\cap \gamma_2(G)\}$,
        \item $\gr_3^G(L)=\gr_3^G(K)+[\gr_2^G(K),\gr_1(G)]$.
    \end{enumerate}
\end{lemma}

\begin{proof}
    Since $K\leq\gamma_2(G)$, we have
    $[K,G]\leq\gamma_3(G)\leq Z(G)$.  For $k\in K$ and $g\in G$,
    \[
    k^g=[g^{-1},k]k\in K[K,G].
    \]
    Thus every conjugate of an element of $K$ belongs to $K[K,G]$.  The
    reverse inclusion follows because the normal closure of $K$ contains both
    $K$ and $[K,G]$.  Hence
    \[
    L=K[K,G]\leq\gamma_2(G).
    \]

    Since $[K,G]\leq\gamma_3(G)$, the images of $L$ and $K$ in
    $\gr_2(G)=\gamma_2(G)/\gamma_3(G)$ coincide.  This proves (2).

    Moreover,
    \[
    L\cap\gamma_3(G)=(K\cap\gamma_3(G))[K,G].
    \]
    For $k\in K$ and $g\in G$, the induced bracket is
    \[
    [k\gamma_3(G),g\gamma_2(G)]
    =[k,g]\gamma_4(G)=[k,g].
    \]
    Therefore the image of $[K,G]$ in $\gr_3(G)$ is exactly
    $[\gr_2^G(K),\gr_1(G)]$, which proves (3).
\end{proof}
The next proposition is a main tool to analyze group structures of exponent $3$ in this section.

\begin{proposition}\label{proposition:gr of free product}
    Let $G_i$ $(i=0,1)$ be groups of exponent $3$.
    \begin{enumerate}
        \item Then the following natural maps induced by the canonical embeddings $G_i\to G_0*_{\mathcal V_3}G_1$ $(i=0,1)$ are isomorphisms.
        \[
        \eta_1:\gr_1(G_0)\oplus \gr_1(G_1) \to \gr_1(G_0*_{\mathcal V_3}G_1)
        \]
        \[
        \eta_2:\gr_2(G_0)\oplus (\gr_1(G_0)\otimes\gr_1(G_1))\oplus \gr_2(G_1) \to \gr_2(G_0*_{\mathcal V_3}G_1)
        \]
        \[
        \eta_3:\gr_3(G_0)\oplus (\gr_2(G_0)\otimes\gr_1(G_1))\oplus (\gr_1(G_0)\otimes\gr_2(G_1)) \oplus \gr_3(G_1) \to \gr_3(G_0*_{\mathcal V_3}G_1)    
        \]
        where $\gr_i(G_0)\otimes \gr_j(G_1)\ni a\otimes b \xmapsto{\eta_{i+j}} [\eta_i(a),\eta_j(b)]\in \gr_{i+j}(G_0*_{\mathcal V_3}G_1)$. 

        \item Let us identify $\mathrm{dom}(\eta_i)$ with $\mathrm{Im}(\eta_i)$. Then, by letting $\mathcal B_{i,j}=\gr_i(G_0)\otimes \gr_j(G_1)$ for $i,j\geq 1$ with $i+j\leq 3$, $\mathcal B_{i,0}=\gr_i(G_0)$,  $\mathcal B_{0,j}=\gr_j(G_1)$, and
        $\mathcal B_{i,j}=0$ $(i+j\geq 4)$,
        \[
            [\mathcal B_{i,j},\mathcal B_{k,l}]\subset \mathcal B_{i+k,j+l}.
            \]
    \end{enumerate}    
\end{proposition}

\begin{proof}
(2) is immediate from the definition of $\eta_i$.
We show (1).
Let $V_n=\gr_1(G_n)$ $(n<2)$ and fix a basis $\mathcal I_n$ of $V_n$ and a representative $a_e\in G_n$ of each vector $e\in \mathcal I_n$.
Let us consider a lift $F_{G_n}\coloneqq F_{\mathcal V_3}(\mathcal I_n)$ of $G_n$, and we will use the
chosen surjective homomorphism $\varphi_n:F_{G_n}\to G_n$ and its kernel
$K_n:=\ker(\varphi_n)\leq\gamma_2(F_{G_n})$ as defined in Proposition \ref{proposition:lift}.
Let $V=V_0\oplus V_1$, $F=F_{G_0}*_{\mathcal V_3} F_{G_1}=F_{\mathcal V_3}(\mathcal I_0\sqcup \mathcal I_1)$, $K=\langle K_0,K_1\rangle_F=K_0K_1$, and $L=\llangle K\rrangle_F$.
(Notice that $\gr_1(F)=V$.)
Then we can identify $G_0*_{\mathcal V_3}G_1$ with $F/L$.

By Lemma \ref{lemma:gr of quotient}, we have a canonical isomorphism $\gr_i(F/L)\cong \gr_i(F)/\gr_i^F(L)$.
Also by Lemma \ref{lemma:gr of normal closure}, 
\[
    \gr_1^F(L)=0, \quad \gr_2^F(L)=\gr_2^F(K), \quad \gr_3^F(L)=\gr_3^F(K)+[\gr_2^F(K),\gr_1(F)].
\]
Now recall that we can identify $\gr_i(F)$ with $\Lambda^i V$ and recall the standard direct sum decomposition of Grassmann algebra:
\begin{eqnarray*}
    \Lambda^1(V_0\oplus V_1)&=&\Lambda^1V_0 \oplus \Lambda^1 V_1 = V_0\oplus V_1,\\ 
    \Lambda^2(V_0\oplus V_1)&=&\Lambda^2V_0 \oplus (V_0\otimes V_1)\oplus\Lambda^2 V_1,\\ 
    \Lambda^3(V_0\oplus V_1)&=&\Lambda^3V_0 \oplus (\Lambda^2V_0\otimes V_1) \oplus (V_0\otimes \Lambda^2V_1)\oplus\Lambda^3 V_1.    
\end{eqnarray*}
Hence,
\begin{eqnarray*}
    \gr_1(G_0*_{\mathcal V_3}G_1) 
    &\cong& \gr_1(F)/\gr_1^F(L)\\ 
    &=& V_{0}\oplus V_{1}.
\end{eqnarray*}
Since $K=K_0K_1$ is abelian, by letting $\gr_2^F(K_n)=R_n\subset \Lambda^2 V_n$,
\begin{eqnarray*}
    \gr_2^F(K)=R_0\oplus R_1 &\subset& \Lambda^2 V_0\oplus \Lambda^2 V_1\subset \Lambda^2 V.
\end{eqnarray*}
Hence $\gr_2^F(K)$ is block-homogeneous with respect to the decomposition.
After we identify $\gr_2(F_{G_n})$ with $\Lambda^2 V_n$, we can consider $\gr_2^{F_{G_n}}(K_n)=\gr_2^F(K_n)=R_n$.
Hence, because $\Lambda^2V_n/R_n \cong \gr_2(F_{G_n}/K_n)=\gr_2(G_n)$,
\begin{eqnarray*}
    \gr_2(G_0*_{\mathcal V_3}G_1) 
    &\cong& \gr_2(F)/\gr_2^F(L)\\ 
    &=& \Lambda^2(V_0\oplus V_1)/\gr_2^F(K)\\ 
    &\cong& \frac{\Lambda^2V_0\oplus (V_0\otimes V_1)\oplus \Lambda^2V_1}{R_0\oplus R_1}\\ 
    &\cong& \gr_2(G_0) \oplus (\gr_1(G_0)\otimes \gr_1(G_1))\oplus \gr_2(G_1).
\end{eqnarray*}

Let $S=\gr_3^F(L)\subset \Lambda^3V$, 
$S_n=\gr_3^{F_{G_n}}(K_n)\subset \Lambda^3V_n$.
Then $\gr_3^F(K)=S_0\oplus S_1$.
Arguing as above,
\[S=S_0+S_1+[R_0+R_1,V]=S_0+S_1+(R_0+R_1)\wedge(V_0+V_1).\]
Because $R_n\wedge V_n\subset S_n$, we have
\[
    S=S_0+R_0\wedge V_1+V_0\wedge R_1+S_1, \quad R_n\wedge V_{1-n}\cong R_n\otimes V_{1-n}.
\]
Hence,
\begin{eqnarray*}
    \gr_3(G_0*_{\mathcal V_3} G_1) 
    &\cong&
    \gr_3(F)/\gr_3^F(L)\\ 
    &\cong&
    \frac{\Lambda^3(V_0)\oplus(\Lambda^2V_0\otimes V_1) \oplus (V_0\otimes \Lambda^2V_1)\oplus\Lambda^3 V_1}{S_0+R_0\otimes V_1+V_0\otimes R_1 + S_1}\\ 
    &\cong&
    \gr_3(G_0)\oplus (\gr_2(G_0)\otimes \gr_1(G_1))\oplus (\gr_1(G_0)\otimes \gr_2(G_1))\oplus \gr_3(G_1),
\end{eqnarray*}
because $S=S_0+R_0\otimes V_1+V_0\otimes R_1+S_1$ is block-homogeneous.
\end{proof}

In the remainder of this section, we will often consider $H=G*_{\mathcal V_3} F_n$ where $F_n=F_{\mathcal V_3}(x_1,\cdots, x_n)$.
By Proposition \ref{proposition:gr of free product}, we use the following identifications:
    \begin{eqnarray*}
        \gr_1(H)&=& \gr_1(G)\oplus \gr_1(F_n)\\ 
        \gr_2(H) &=& \gr_2(G)\oplus (\gr_1(G)\otimes \gr_1(F_{n})) \oplus \gr_2(F_{n})\\
        \gr_3(H) &=& \gr_3(G)\oplus (\gr_2(G)\otimes \gr_1(F_{n})) \oplus (\gr_1(G)\otimes \gr_2(F_{n})) \oplus \gr_3(F_{n})
    \end{eqnarray*}
    and, by Proposition \ref{proposition:gr(F) is Grassmann algebra},
    \[
        \gr_i F_{n} = \Lambda^i V_n
        \]
    where $V_n$ is a vector space over $\mathbb F_3$ with basis $\{x_j\mid 1\leq j\leq n\}$.

\begin{lemma}\label{lemma:free-product-amalgam}
    If $D\leq_{\LCS}G\in\mathcal V_3$, then for every $B\in\mathcal V_3$,
    a natural homomorphism
    $D*_{\mathcal V_3}B\to G*_{\mathcal V_3}B$ induced by the inclusion
    $D\leq G$ is injective.
\end{lemma}
\begin{proof}
    Since $D\leq_\LCS G$, a natural homomorphism $\gr(D) \to \gr(G)$ is injective by Proposition \ref{proposition:gr(f) and LCS}.
    Hence, by Proposition \ref{proposition:gr of free product}, $\gr(D*_{\mathcal V_3}B) \to \gr(G*_{\mathcal V_3}B)$ is injective.
    Again, by Proposition \ref{proposition:gr(f) and LCS},
    a natural homomorphism
    $D*_{\mathcal V_3}B\to G*_{\mathcal V_3}B$ is injective. 
\end{proof}

\begin{lemma}\label{lemma:coincidence of central series}
    Let $1\neq G\in \mathcal V_3$ and let $F_2=F_{\mathcal V_3}(x,y)$ be the relatively free group of rank $2$ in $\mathcal V_3$.
    Then $G\leq_{\LCS}G*_{\mathcal V_3}F_2$, and the lower central series of
    $G*_{\mathcal V_3}F_2$ coincides with its upper central series in reverse
    order.
\end{lemma}

\begin{proof}
    Let $H=G*_{\mathcal V_3}F_2$.
    Let $f:G\to H$ be a natural inclusion map. By Proposition \ref{proposition:gr of free product}, $\gr(f):\gr G\to \gr H$ is injective. Hence $G\leq_\LCS H$ by Proposition \ref{proposition:gr(f) and LCS}.

    \begin{claim}
        $\gamma_2(H)=Z_2(H)$.
    \end{claim}
    Let $a\in H\setminus\gamma_2(H)$. We show that there is $b,c\in H$ such that $[a,b,c]\neq 1$, which implies $a\not\in Z_2(H)$.
    Because $\gr_1(H)=\gr_1(G)\oplus \gr_1(F_2)$, $\gr_1(a)=\alpha+\beta$ for some $\alpha\in \gr_1(G)$ and $\beta\in \gr_1(F_2)=\Lambda^1V_2$ where $V_2$ is a $\mathbb F_3$-vector space with a basis $\{x,y\}$.
    
    Case 1: Suppose $\alpha\neq 0$. Then $\gr_3([a,x,y])=-\alpha\otimes (x\wedge y)+\beta\wedge x\wedge y\neq 0$, hence $[a,x,y]\neq 1$.

    Case 2: Suppose $\alpha=0$ and $\beta \neq 0$. 
    Take $g \in G\setminus \gamma_2(G)$. 
    Since $\dim \gr_1(F_2)=2$, we can find  $w\in F_2$ such that $\{\beta, \gr_1(w)\}$ is linearly independent in $\gr_1(F_2)$.
    Then $[a,g,w]\neq 1$ because $\gr_3([a,g,w])=\gr_1(g) \otimes (\beta\wedge \gr_1(w)) \neq 0$.

    \begin{claim}
        $\gamma_3(H)=Z(H)$
    \end{claim}
    Let $a\not\in\gamma_3(H)$. We show that there is $b\in H$ such that $[a,b]\neq 1$, which implies $a\not\in Z(H)$.
    Because of Claim A, we can assume $a\in \gamma_2(H)$.
    Since $\gr_2(H)=\gr_2(G)\oplus (\gr_1(G)\otimes \gr_1(F_2))\oplus \gr_2(F_2)$,
    we have $\gr_2(a)=\alpha_2+ \alpha_1\otimes x+\alpha_1'\otimes y + \beta_2$ for some $\alpha_2\in\gr_2(G)$, $\alpha_1,\alpha_1'\in\gr_1(G)$, and $\beta_2\in\gr_2(F_2)$.
    (Here, we used that $\gr_1(F_2)=\mathbb F_3x\oplus \mathbb F_3 y$.)
    
    Case 1: If $\alpha_2\neq 0$, then $[a,x]\neq 1$.
    
    Case 2: If $\alpha_2=0$ and $\alpha_1\neq 0$, then $\gr_3([a,y])=-\alpha_1\otimes(x\wedge y)+\beta_2\wedge y=-\alpha_1\otimes(x\wedge y)\neq 0$. In the case of $\alpha_1'\neq 0$, use $x$ instead of $y$.

    Case 3: If $\alpha_2=0$, $\alpha_1=\alpha_1'=0$ and $\beta_2\neq 0$, then choose $g\in G\setminus \gamma_2(G)$. We have $[a,g]\neq 1$ because $\gr_3([a,g])=-\gr_1(g)\otimes \beta_2\neq 0$.
\end{proof}

\begin{lemma}\label{lemma:basic commutator root}
    Let $G\in \mathcal V_3$, $g_1,\cdots, g_n \in \gamma_2(G)$ and $F_{2n}=F_{\mathcal V_3}(x_1,y_1,\cdots, x_n,y_n)$.
    Then a natural homomorphism $G\to H$ is injective where 
    $$
        H=(G*_{\mathcal V_3}F_{2n})/
        \llangle g_i^{-1}[x_i,y_i]\mid 1\leq i\leq n\rrangle.
    $$
\end{lemma}

\begin{proof}
    Let 
    \[ 
        H_0=G*_{\mathcal V_3}F_{2n},\quad 
        r_i = g_i^{-1}[x_i,y_i] \in H_0
         \] 
    \[ 
        K=\langle r_i\mid i\leq n\rangle_{H_0}\leq \gamma_2(H_0), \quad
        L=\llangle K\rrangle_{H_0}=\llangle r_1\rrangle_{H_0}\cdots\llangle r_n\rrangle_{H_0}, \quad H= H_0/L.
    \]  

    \begin{claim}
        $L\leq \gamma_2(H_0)$ and $\llangle r_i\rrangle = \{r_i^m[r_i,h]\mid m\in \mathbb F_3, h\in H_0\}$.
    \end{claim}
    Since $L=K[K,H_0]$ and $K\leq \gamma_2(H_0)$, we have $L\leq \gamma_2(H_0)$.
    Similarly,
    \[
        \llangle r_i\rrangle
        =\langle r_i\rangle [\langle r_i\rangle,H_0]
        =\langle r_i\rangle [r_i,H_0].
    \]
    Here, we used that $[r^{-1},h]=[r,h]^{-1}=[r,h^{-1}]$.
    Since $r_i\in\gamma_2(H_0)$, $[r_i,h][r_i,h']=[r_i,hh']$.
    Hence $[r_i,H_0]=\{[r_i,h]\mid h\in H_0\}$.

    It suffices to show that $L\cap G=1$, so let $a\in L\cap G$.
    By Claim A, $a=\prod_i r_i^{m_i}[r_i,h_i]\in\gamma_2(H_0)$ for some $h_i\in H_0$.

    \begin{claim}
    $m_i=0$ $(i\leq n)$. Hence $a=\prod_i[r_i,h_i] \in\gamma_3(H_0)$.
    \end{claim}
    Since $g_i,r_i,a\in \gamma_2(H_0)$, we can calculate that
    \begin{eqnarray*}
    \gr_2(a)&=&\sum_i m_i\gr_2(r_i)\\ 
    &=&\sum_i m_i(\gr_2([x_i,y_i])-\gr_2(g_i))\\ 
    &=& \sum_i m_i(x_i\wedge y_i) - \sum_i m_i\gr_2(g_i).
    \end{eqnarray*}
    Because $a\in G$, $\gr_2(a)$ must be in $\gr_2(G)$ whereas $0\neq x_i\wedge y_i\in \gr_2(F_{2n})$ and $\{x_i\wedge y_i\mid i\leq n\}$ is linearly independent.
    Hence we have $m_i=0$ for all $i$ and it follows that $a\in\gamma_3(H_0)$.

    \begin{claim}
        $a=1$.
    \end{claim}
    Recall that $a=\prod_i [r_i,h_i]\in \gamma_3(H_0)$.
    Suppose $\gr_1(h_i)=s_i+ t_i$ where $s_i\in \gr_1(G)$ and $t_i\in \gr_1(F_{2n})$.
    Then, 
    \begin{eqnarray*}
        \gr_3(a) 
        &=& \sum_i [x_i\wedge y_i-\gr_2(g_i),s_i+t_i]\\ 
        &=& \sum_i -s_i\otimes(x_i\wedge y_i) + x_i\wedge y_i\wedge t_i - [\gr_2(g_i),s_i] - \gr_2(g_i)\otimes t_i.
    \end{eqnarray*}
    Because $a\in G$ and $\gr_3(a)\in\gr_3(G)$,
    $\gr_3(a)$ must be $-\sum_i[\gr_2(g_i),s_i]$. Hence $\sum_i s_i\otimes (x_i\wedge y_i)=0$, which implies $s_i=0$.
    Therefore we have $\gr_3(a)=0$ and $a=1$.
\end{proof}

\begin{lemma}\label{lemma:commutator root}
Let $C\leq G\in \mathcal V_3$ such that $C$ is generated by $m$ elements. Then there is an extension $G\leq H\in \mathcal V_3$ and $C\leq D \leq H$ such that
\begin{enumerate}
    \item $D$ is generated by $C$ and at most $2m$ elements in $H$, and
    \item $D\cap \gamma_2(H)= \gamma_2(D)$.
\end{enumerate}    
\end{lemma}

\begin{proof}
    Let $g_1,\cdots, g_n\in C\cap \gamma_2(G)$ be such that $\{g_i\gamma_2(C)\mid i\leq n\}$ is a basis of $(C\cap \gamma_2(G))/\gamma_2(C)$.
    Then $C\cap \gamma_2(G)\subset \langle g_i\mid i\leq n\rangle \gamma_2(C)$ and $n\leq \dim\gr_1(C) \leq m$.
    Let $F_{2n}=F_{\mathcal V_3}(x_1,y_1,\cdots, x_n,y_n)$, $H_0=G*_{\mathcal V_3} F_{2n}$ and $H=H_0/L$ where $L=\llangle g_i^{-1}[x_i,y_i]\mid i\leq n\rrangle$.
    Then, by Lemma \ref{lemma:basic commutator root}, $G\leq H$.

    \begin{claim}
        $D=\langle C, x_i,y_i\mid i\leq n\rangle_H \leq H$ satisfies $D\cap \gamma_2(H)=\gamma_2(D)$.
    \end{claim}
    First notice that, since $L\leq \gamma_2(H_0)$, $\gr_1(H)=\gr_1(H_0/L)\cong \gr_1(H_0)=\gr_1(G)\oplus \gr_1(F_{2n})$.

    Let $a\in D\cap \gamma_2(H)$.
    Then $a\in c\prod x_i^{m_i}y_i^{n_i}\gamma_2(D)$ for some $c\in C$.
    Hence, $0=\gr_1(a)=\gr_1(c)+\sum_i m_ix_i+\sum_i n_iy_i$.
    Because $\{x_i,y_i\mid i\leq n\}$ is a basis of $\gr_1(F_{2n})$, $m_i=n_i=0$ for all $i\leq n$ and $\gr_1(c)=0$. 
    Therefore $c\in \gamma_2(G)$ follows.
    This argument shows that $D\cap \gamma_2(H)=(C\cap \gamma_2(G))\gamma_2(D)$.
    By the choice of $g_i$ $(i\leq n)$, 
    \begin{eqnarray*}
        D\cap \gamma_2(H)&=& (C\cap \gamma_2(G))\gamma_2(D)\\ 
        &\subset& \left(\langle g_i\mid i\leq n\rangle\gamma_2(C)\right)\gamma_2(D)\\ 
        &\subset& \gamma_2(D). 
    \end{eqnarray*}
\end{proof}

\begin{lemma}\label{lemma:triple commutator root}
    Let $G\in\mathcal V_3$, $g_1,\cdots, g_n \in Z(G)$ and let $F_{3n}=F_{\mathcal V_3}(x_1,y_1,z_1,\cdots, x_n,y_n,z_n)$.
    Then
    the natural map $G\to H$ is injective where
    \[
        H=(G\times F_{3n})/\llangle g_i^{-1}[x_i,y_i,z_i]\mid i\leq n\rrangle
    \]
\end{lemma}

\begin{proof}
    Let $L=\llangle g_i^{-1}[x_i,y_i,z_i]\mid i\leq n\rrangle\leq G\times F_{3n}$.
    Since $g_i\in Z(G)$, $g_i\in Z(G\times F_{3n})$.
    Hence $L=\langle g_i^{-1}[x_i,y_i,z_i]\mid i\leq n\rangle$.

    We prove that $G\cap L=1$ in $G\times F_{3n}$.
    Let $a\in G\cap L$.
    Then $a=\prod_i g_i^{-m_i}[x_i,y_i,z_i]^{m_i}$.
    Since $a\in G$, $m_i=0$ for all $i\leq n$. 
    This means $a=1$.
\end{proof}

\begin{lemma}\label{lemma:number of generators for triple commutator roots}
    Let $C\leq G\in\mathcal V_3$ and let $C$ be generated by at most $n$ elements.
    Suppose that $C\cap \gamma_2(G)=\gamma_2(C)$.
    Then there are $G\leq H$ and $C\leq D\leq H$ such that
    \begin{enumerate}
        \item $D$ is generated by $C$ and at most $3n(n-1)/2$ elements in $H$,
        \item $D\leq_\LCS H$.
    \end{enumerate}
\end{lemma}
\begin{proof}
    Let $g_0,\cdots, g_{m-1} \in C\cap \gamma_3(G)$ be such that $\{g_i\gamma_3(C)\mid i< m\}$ generates the $\mathbb F_3$-vector space $(C\cap \gamma_3(G))/\gamma_3(C)$.
    Since $C\cap \gamma_3(G)\leq C\cap\gamma_2(G)=\gamma_2(C)$, we can choose $m\leq \dim\gr_2(C)\leq \binom n2$.
    By Lemma \ref{lemma:triple commutator root}, $H=(G\times F_{3m})/L$ is an extension of $G$  where $L=\langle g_i^{-1}[x_i,y_i,z_i]\mid i< m\rangle$.
    Let $D=\langle C,x_i,y_i,z_i\mid i< m\rangle_H$. Then $D=(C\times F_{3m})/L$.
    \begin{claim}
        $\gamma_2(H)\cap D=\gamma_2(D)$
    \end{claim}
    Since $g_i\in C\cap \gamma_3(G)\leq \gamma_2(C)$, $L\leq \gamma_2(C\times F_{3m})$.
    Also, $\gamma_2(H)=(\gamma_2(G\times F_{3m})L)/L=\gamma_2(G\times F_{3m})/L=(\gamma_2(G)\times \gamma_2(F_{3m}))/L$.
    Hence
    \begin{eqnarray*}
        \gamma_2(H)\cap D 
        &=& \frac{\gamma_2(G)\times \gamma_2(F_{3m})}L \quad \cap \quad   \frac{C\times F_{3m}}L \\ 
        &=& \frac{(\gamma_2(G)\cap C)\times \gamma_2(F_{3m})}L \\ 
        &=& \frac{\gamma_2(C)\times \gamma_2(F_{3m})}L \\
        &=& \frac{\gamma_2(C\times F_{3m})}L \\ 
        &=& \gamma_2(D).
    \end{eqnarray*}
    \begin{claim}
        $\gamma_3(H)\cap D=\gamma_3(D)$
    \end{claim}
    Since $g_i\in \gamma_3(G)\cap C$ and $[x_i,y_i,z_i]\in \gamma_3(F_{3m})$, $L\leq (\gamma_3(G)\cap C)\times \gamma_3(F_{3m})\leq \gamma_3(G\times F_{3m})$.
    Hence,
    \begin{eqnarray*}
        \gamma_3(H) &=& \frac{\gamma_3(G\times F_{3m})}L, \\ 
        \gamma_3(H) \cap D &=& \frac{(\gamma_3(G)\cap C)\times\gamma_3(F_{3m})}L,\\ 
        \gamma_3(D)&=&\frac{(\gamma_3(C)\times \gamma_3(F_{3m}))L}L.
    \end{eqnarray*}
    Suppose $a\in (\gamma_3(G)\cap C)\times \gamma_3(F_{3m})$.
    By the choice of $g_i$ $(i<m)$, $a=h\prod_i g_i^{m_i} f$ for some $h\in \gamma_3(C)$, $m_i\in\mathbb F_3$, and $f\in\gamma_3(F_{3m})$.
    Therefore, $a\in h(\prod_i [x_i,y_i,z_i]^{m_i})f L$.
    This shows that $\gamma_3(H)\cap D\subset \gamma_3(D)$.
\end{proof}

\begin{remark}
    The bound on the number of additional generators in the above lemma is not optimal.
    Indeed, for example, $F_{\mathcal V_3}(x_1,\cdots, x_4)$ contains $\binom 43=4$ elements $[x_i,x_j,x_k]$ $(i<j<k\leq 4)$ that are linearly independent in $\gr_3(F_{4})$.
    Hence, for $g_1,\cdots, g_4$, $L=\langle g_1^{-1}[x_1,x_2,x_3],\cdots, g_4^{-1}[x_2,x_3,x_4]\rangle$ is enough to prove the desired conditions.
    This trick works for Lemma \ref{lemma:commutator root}. 
\end{remark}

\begin{proposition}\label{proposition:structure of e.c. model}
    Let $M$ be an e.c. model of $T_3$. Then the lower central series of $M$
    coincides with its upper central series in reverse order.
    Moreover,
    \[
        \gamma_2(M)=\{[a,b]\mid a,b\in M\},
    \]
    \[
        \gamma_3(M)=\{[a,b,c]\mid a,b,c\in M\}.
    \]
\end{proposition}

\begin{proof}
    Let $a\in \gamma_2(M)$. By Lemma \ref{lemma:basic commutator root}, there is an extension $M\leq H$ and $b',c'\in H$ such that $[b',c']=a$.
    Since $M$ is e.c., we can find $b,c\in M$ such that $[b,c]=a$ in $M$.
    This proves $\gamma_2(M)=\{[a,b]\mid a,b\in M\}$.
    A similar argument, using Lemma \ref{lemma:triple commutator root}, proves the corresponding equality for $\gamma_3(M)$.
    
    Let $a\in M\setminus\gamma_2(M)$ and let $H=M*_{\mathcal V_3}F_2$.
    Since $M\leq_\LCS H$, $a\not\in \gamma_2(H)$.
    Hence there are $b',c'\in H$ such that $[a,b',c']\neq 1$. 
    Since $M$ is an e.c. model, we can find $b,c\in M$ such that $[a,b,c]\neq 1$, which proves $\gamma_2(M)=Z_2(M)$.
    A similar argument proves $\gamma_3(M)=Z(M)$.
\end{proof}

\begin{proposition}\label{proposition:bdd LCS}
    Let $M$ be an e.c. model of $T_3$ and let $C\leq M$ be a finite
    subgroup.
    Then there is $C\leq D\leq M$ such that 
    \begin{enumerate}
        \item $D$ is generated by at most $15n^2$ elements where $C$ is generated by at most $n$ elements,
        \item the lower central series of $D$ coincides with its upper central series in reverse order. In particular, $D\leq_\LCS M$. 
    \end{enumerate}
\end{proposition}

\begin{proof}
    If $C=1$, take $D=1$. Hence we may assume that $C\neq 1$.
    By Lemma \ref{lemma:commutator root}, there is an
    extension $M\leq H_1$, $C\leq D_1\leq H_1$ such that 
    $D_1\cap\gamma_2(H_1)=\gamma_2(D_1)$.
    The group $D_1$ is generated by $C$ and at most $2n$ elements.  
    Hence $D_1$ is generated by at most $3n$ elements.
    Next we apply Lemma
    \ref{lemma:number of generators for triple commutator roots} and get an extension
    $H_1\leq H_2$ and $D_1\leq D_2\leq_\LCS H_2$.
    $D_2$ is generated by $D_1$ and at most $3(3n)(3n-1)/2$ elements.
    Hence $D_2$ is generated by at most $3n+9n(3n-1)/2$ elements.

    Let $F_2=F_{\mathcal V_3}(\{x,y\})$ and set
    \[
    D_3=D_2*_{\mathcal V_3}F_2,
    \qquad
    H=H_2*_{\mathcal V_3}F_2.
    \]
    Since $1\neq D_2\leq_{\LCS}H_2$, Lemma
    \ref{lemma:free-product-amalgam} implies that $D_3\leq H$.
    By Lemma \ref{lemma:coincidence of central series}, the lower central
    series of $D_3$ coincides with its upper central series in reverse order.
    $D_3$ is generated by at most
    \[
    3n+\frac{9n(3n-1)}2+2\leq 15n^2 
    \]
    elements.
    Since $M$ is an e.c. model and $D_3$ is finite, we can find $C\leq D\leq M$ such that $D\cong_C D_3$. 
    Since $D$ is a copy of $D_3$, the lower central series of $D$ coincides with its upper central series in reverse order.
\end{proof}

\section{Examples}
\label{sec:examples}
In this section, we give examples showing that lower-central strictness and
existential closedness cannot be omitted from the relevant results
in Sections \ref{sec:main-results} and \ref{sec:structure}.

We start with an easy example.
\begin{example}
    Let $G=F_{\mathcal V_3}(x,y,z)=B(3,3)$.
    Then by Lemma \ref{lemma:coincidence of central series} (or by a direct calculation), $\gamma_2(G)=Z_2(G)$ and $\gamma_3(G)=Z(G)$.
    
    Let $a=[x,y]\in \gamma_2(G)$ and $A=\langle a,z\rangle \leq G$.
    Then, by Fact \ref{fact:Levi and van der Waerden}, $A\cong B(2,3)$ and $a\not\in \gamma_2(A)=\{[a,z]^k\mid k\in \mathbb F_3\}$. 
    Let $F_2=F_{\mathcal V_3}(v,w)$. 
    Then
    \begin{enumerate}
        \item $A\not\leq_\LCS G$,
        \item The groups $A*_{\mathcal V_3}F_2$ and $G$ do not have an amalgam over $A$ in $T_3$, because $[a,v,w]\neq 1$ in $A*_{\mathcal V_3}F_2$ but $[a,b,c]=[[x,y],b,c]=1$ for any $b,c\in H\geq G$,
        \item $A*_{\mathcal V_3} F_2 \to G*_{\mathcal V_3}F_2$ is not an embedding, because $[a,v,w]\mapsto 1$.
    \end{enumerate}
\end{example}

\begin{remark}
The example shows that lower-central strictness cannot be omitted
from Lemma \ref{lemma:free-product-amalgam}.
In the proof of Theorem \ref{thm:main}, Proposition A supplies a
subgroup for which this injectivity holds, allowing the witness to
failure of amalgamation to be transferred to a finite coproduct.    
\end{remark}

We next show that the uniform bounds in Theorem \ref{thm:main}
and Proposition \ref{proposition:bdd LCS} fail without
existential closedness.

\begin{lemma}\label{lemma:D can be large}
    Let $F_\omega=F_{\mathcal V_3}(x_i\mid i\in\mathbb N)\in \mathcal V_3$.
    Let $n\geq 1$, $D\leq F_\omega$, and let $a_n:=\prod_{i<n} [x_{2i},x_{2i+1}] \in D\leq F_\omega$.
    If $a_n\in \gamma_2(D)$, then $d(D)\geq 2n$.
\end{lemma}   
\begin{proof}
    Put $V=\gr_1(F_\omega)=\bigoplus_i \mathbb F_3x_i$ and $U=\gr_1^{F_\omega}(D)\leq V$.
    The image of $\gamma_2(D)$ in
    $\gr_2(F_\omega)=\Lambda^2V$
    is contained in $\Lambda^2U$, since $\gamma_2(D)$ is generated
    by commutators of elements of $D$.
    Hence $a_n\in\gamma_2(D)$ implies
    \[
        \alpha_n:=\gr_2(a_n)
        =\sum_{i=0}^{n-1}x_{2i}\wedge x_{2i+1}
        \in\Lambda^2U.
    \]
    Let $V^*$ be the vector space of all linear maps from $V$ to $\mathbb F_3$.
    For $\lambda\in V^*$, define the contraction
    $\iota_\lambda:\Lambda^2V\to V$ by
    \[
        \iota_\lambda(u\wedge v)=\lambda(u)v-\lambda(v)u.
    \]
    This is well-defined and linear, and
    $\iota_\lambda(\Lambda^2U)\subseteq U$.
    Let $x_j^*\in V^*$ be the coordinate functional defined by
    $x_j^*(x_i)=\delta_{ij}$. For $0\leq k<n$, we obtain
    \[
        x_{2k+1}=\iota_{x_{2k}^*}(\alpha_n)\in U,
        \qquad
        -x_{2k}=\iota_{x_{2k+1}^*}(\alpha_n)\in U.
    \]
    Hence $\dim U\geq2n$.
    Since the images in $U$ of any generating set of $D$ span $U$,
    we have $d(D)\geq \dim U\geq 2n$.
\end{proof}

\begin{lemma}\label{lemma:amalgam over the cyclic group}
    Let $1\neq a\in G\in\mathcal V_3$ and $A=\langle a \rangle_G$.
    Let $F_3=F_{\mathcal V_3}(s,x,y)$ and identify $A$ with $\langle s\rangle$ via $a\mapsto s$.
    Then $G$ and $F_3$ have an amalgam over $A$ if and only if $a\not\in \gamma_2(G)$.
\end{lemma}
\begin{proof}
    Suppose that $a\in \gamma_2(G)$. Then for any $H\in \mathcal V_3$ containing $G$, $[a,b,c]=1$ for every $b,c\in H$.
    On the other hand, $[a,x,y]\neq 1$ in $F_3$. This means $G$ and $F_3$ do not have an amalgam $H$ over $A$.

    Conversely, suppose that $a\not\in \gamma_2(G)$.
    There is a linear map
    \[ 
        \lambda:\gr_1(G)\to\mathbb F_3
    \]
    sending $\gr_1(a)$ to $1$.
    Hence
    $$ r_G(g):=a^{\lambda(\gr_1(g))} $$
    defines a retraction $r_G:G\to A$.
    There is also a retraction $r_{F_3}:F_3\to A$ $(x,y\mapsto 1$ and $s\mapsto a)$.
    Hence $G\times F_3$ is an amalgam of $G$ and $F_3$ over $A$ through $G\ni g\mapsto (g,r_G(g))\in G\times F_3$ and $F_3\ni b\mapsto (r_{F_3}(b), b) \in G\times F_3$.
\end{proof}

Let $F_\omega=F_{\mathcal V_3}(\{x_i\mid i\in\mathbb N\})\in \mathcal V_3$, $n\geq 1$, $a_n=\prod_{i<n}[x_{2i},x_{2i+1}]$ and let $A_n=\langle a_n\rangle \leq F_\omega$.
(Clearly, $A_n\cong A_m$ for every $n,m>0$ because $A_n$ is non-trivial and $a_n^3=1$.)
Let $F_3=F_{\mathcal V_3}(s,y,z)\in\mathcal V_3$ and identify $A_n$ with $\langle s\rangle\leq F_3$ via $a_n\mapsto s$. Then $F_3$ and $F_\omega$ do not have an amalgam over $A_n$ by the previous lemma.

\begin{proposition}
    \begin{enumerate}
        \item If $A_n\leq D\leq_\LCS F_\omega$ then $d(D)\geq 2n$.
        \item 
        For any $A_n\leq D\leq F_\omega$, if $D$ and $F_3$ do not have an amalgam in $\mathcal V_3$ over $A_n$, then $d(D)\geq 2n$. 
    \end{enumerate}
\end{proposition}

\begin{proof}
    (1): Since $a_n\in \gamma_2(F_\omega)$, if $D\leq_\LCS F_\omega$, then $a_n\in \gamma_2(D)$.
    Hence, by Lemma \ref{lemma:D can be large}, $d(D)\geq 2n$.

    (2): Since $F_3$ and $D$ do not have an amalgam over $A_n$, by Lemma \ref{lemma:amalgam over the cyclic group}, $a_n\in \gamma_2(D)$.
    Hence, by Lemma \ref{lemma:D can be large}, $d(D)\geq 2n$.
\end{proof}
Since $d(A_n)=1$ and $d(F_3)=3$ for every $n\geq1$, while the
required numbers of generators are at least $2n$, neither the
finite-envelope bound in Proposition~\ref{proposition:bdd LCS} nor
the bounded-obstruction conclusion of Theorem~\ref{thm:main} extends
to arbitrary models of $T_3$.




\begin{thebibliography}{dEMRS2025}

\bibitem[Adi1979]{Adian1979}
S.~I.~Adian,
\emph{The Burnside Problem and Identities in Groups},
translated from the Russian by J.~Lennox and J.~Wiegold,
Ergebnisse der Mathematik und ihrer Grenzgebiete, vol.~95,
Springer-Verlag, Berlin--New York, 1979.

\bibitem[CK1990]{ChangKeisler1990}
C.~C.~Chang and H.~J.~Keisler,
\emph{Model Theory},
third edition,
Studies in Logic and the Foundations of Mathematics, vol.~73,
North-Holland, Amsterdam, 1990.

\bibitem[dEMRS2025]{dElbeeMuellerRamseySiniora2025}
C.~d'Elb\'ee, I.~M\"uller, N.~Ramsey, and D.~Siniora,
\emph{Model-theoretic properties of nilpotent groups and Lie algebras},
J. Algebra \textbf{662} (2025), 640--701.
\href{https://doi.org/10.1016/j.jalgebra.2024.08.012}
     {doi:10.1016/j.jalgebra.2024.08.012}.

\bibitem[Ekl1972]{Eklof1972}
P.~C.~Eklof,
\emph{Some model theory of abelian groups},
J. Symbolic Logic \textbf{37} (1972), no.~2, 335--342.
\href{https://doi.org/10.2307/2272976}
     {doi:10.2307/2272976}.


\bibitem[ES1971]{EklofSabbagh1971}
P.~C.~Eklof and G.~Sabbagh,
\emph{Model-completions and modules},
Ann. Math. Logic \textbf{2} (1971), no.~3, 251--295.
\href{https://doi.org/10.1016/0003-4843(71)90016-7}
     {doi:10.1016/0003-4843(71)90016-7}.

\bibitem[FK2025]{FracekKowalski2025}
M.~Fr\k{a}cek and P.~Kowalski,
\emph{Some model theory of the Heisenberg group},
preprint (2025).
\href{https://arxiv.org/abs/2512.09414}{arXiv:2512.09414}.

\bibitem[HKTY2023]{HoffmannKowalskiTranYe2023}
D.~M.~Hoffmann, P.~Kowalski, C.-M.~Tran, and J.~Ye,
\emph{Of model completeness and algebraic groups},
preprint (2023).
\href{https://arxiv.org/abs/2312.08988}{arXiv:2312.08988}.


\bibitem[LW1933]{LeviVanDerWaerden1933}
F.~Levi and B.~L.~van der Waerden,
\emph{\"Uber eine besondere Klasse von Gruppen},
Abh. Math. Sem. Univ. Hamburg \textbf{9} (1933), 154--158.
\href{https://doi.org/10.1007/BF02940639}
     {doi:10.1007/BF02940639}.

\bibitem[Mai1989]{Maier1989}
B.~J.~Maier,
\emph{On nilpotent groups of exponent $p$},
J. Algebra \textbf{127} (1989), no.~2, 279--289.
\href{https://doi.org/10.1016/0021-8693(89)90253-6}
     {doi:10.1016/0021-8693(89)90253-6}.

\bibitem[NA1968]{NovikovAdian1968}
P.~S.~Novikov and S.~I.~Adian,
\emph{Infinite periodic groups. I--III},
Math. USSR-Izv. \textbf{2} (1968),
no.~1, 209--236;
no.~2, 241--479;
no.~3, 665--685.
\href{https://doi.org/10.1070/IM1968v002n01ABEH000637}
     {doi:10.1070/IM1968v002n01ABEH000637};
\href{https://doi.org/10.1070/IM1968v002n02ABEH000640}
     {doi:10.1070/IM1968v002n02ABEH000640};
\href{https://doi.org/10.1070/IM1968v002n03ABEH000653}
     {doi:10.1070/IM1968v002n03ABEH000653}.

\bibitem[Sar1974]{Saracino1974}
D.~Saracino,
\emph{Wreath products and existentially complete solvable groups},
Trans. Amer. Math. Soc. \textbf{197} (1974), 327--339.
\href{https://doi.org/10.1090/S0002-9947-1974-0342391-5}
     {doi:10.1090/S0002-9947-1974-0342391-5}.

\bibitem[Sar1976]{Saracino1976}
D.~Saracino,
\emph{Existentially complete nilpotent groups},
Israel J. Math. \textbf{25} (1976), 241--248.
\href{https://doi.org/10.1007/BF02757003}
     {doi:10.1007/BF02757003}.

\bibitem[SW1979]{SaracinoWood1979}
D.~Saracino and C.~Wood,
\emph{Periodic existentially closed nilpotent groups},
J. Algebra \textbf{58} (1979), no.~1, 189--207.
\href{https://doi.org/10.1016/0021-8693(79)90199-6}
     {doi:10.1016/0021-8693(79)90199-6}.

\bibitem[Tak2022]{Takeuchi2022}
K.~Takeuchi,
\emph{On model companions of some classes of groups},
RIMS K\^oky\^uroku \textbf{2218} (2022), 79--84.



\end{thebibliography}


\end{document}